\documentclass[11pt]{amsart}
\usepackage{amscd,amssymb,float}
\input xy
\xyoption{all}

\newtheorem{theorem}{Theorem}[section]

\newtheorem{proposition}[theorem]{Proposition}

\renewcommand{\leq}{\leqslant}
\renewcommand{\geq}{\geqslant}

\theoremstyle{definition}

\theoremstyle{definition}
\newtheorem{remark}[theorem]{Remark}
\numberwithin{equation}{section}

\numberwithin{equation}{section} \numberwithin{figure}{section}

\author{Yuguang Zhang}
 \thanks{The  author is supported  by   the  Simons Foundation's program  Simons Collaboration on Special Holonomy in Geometry, Analysis and Physics.}
\address{University of Bath}
\email{yuguangzhang76@yahoo.com}
\title{Note on  equivalences  for degenerations of Calabi-Yau manifolds}

\begin{document}
\begin{abstract}
This note  studies  the equivalencies  among  convergences  of Ricci-flat  K\"{a}hler-Einstein  metrics
on  Calabi-Yau manifolds,  cohomology classes and potential functions.
\end{abstract}
\maketitle

\section{Introduction }
A Calabi-Yau manifold $X$   is a   simply connected complex  projective manifold with trivial canonical bundle $\mathcal{K}_{X}\cong \mathcal{O}_{X}$, and the Hodge numbers  $h^{n,0}(X)=1 $, $h^{i,0}(X)=0$, $0<i<n$.  A polarized Calabi-Yau manifold $(X,L)$ is a Calabi-Yau manifold $X$ with an ample line bundle $L$.

 In  \cite{Ya1}, S.-T. Yau  proved the  Calabi's conjecture
 which asserts the existence of Ricci-flat K\"{a}hler-Einstein  metrics on Calabi-Yau manifolds.
More explicitly,
  if $X$ is a Calabi-Yau manifold, then
 for  any K\"{a}hler class $\alpha\in H^{1,1}(X, \mathbb{R})$,
there exists a
     unique Ricci-flat K\"{a}hler-Einstein  metric
     $\omega\in\alpha$, i.e. $${\rm Ric}(\omega)\equiv 0.$$
The  Riemannian holonomy group of such metric  is   $SU(n)$.  Conversely, one can show that  a simply connected compact   Riemannian manifold with holonomy group   $SU(n)$ is a  Calabi-Yau manifold in  our definition, and the metric is a  Ricci-flat K\"{a}hler-Einstein  metric  (cf.  \cite{Yau2,GHJ}).

 A degeneration of  Calabi-Yau $n$-manifolds $\pi: \mathcal{X}\rightarrow \Delta$  is a flat morphism from a  variety $\mathcal{X}$ of dimension $n+1$  to a disc $ \Delta\subset \mathbb{C}$ such that for any $t\in \Delta^{*}=\Delta \backslash\{0\}$,  $X_{t}=\pi^{-1}(t)$ is a Calabi-Yau manifold,  and  the central fiber   $X_{0}=\pi^{-1}(0)$ is singular. If  there is a  relative  ample line bundle  $\mathcal{L}$   on $\mathcal{X}$, we call it a  degeneration  of polarized Calabi-Yau manifolds, denoted by   $(\pi: \mathcal{X}\rightarrow \Delta, \mathcal{L})$. A Calabi-Yau variety $X_0$ is a normal projective Gorenstein variety with trivial canonical    sheaf $\mathcal{K}_{X_0}\cong \mathcal{O}_{X_0}$, and having at worst   canonical singularities, i.e.  for any resolution $\bar{\pi}: \bar{X}_0\rightarrow X_0$,   $$\mathcal{K}_{\bar{X}_0} \cong_{\mathbb{Q}}  \bar{\pi}^{*}\mathcal{K}_{X_0}+\sum\limits_{E} a_{E}E, \ \ {\rm and} \  \  a_{E} \geq 0, $$ where $E$ are effective  exceptional prime divisors, and $\cong_{\mathbb{Q}} $ stands for the $\mathbb{Q}$-linear equivalence.

  There are several recent works, due to C.-L. Wang, V. Tosatti, and S. Takayama etc. (\cite{Wang1,Wang2,To,Ta}),   establishing  the equivalencies of various properties along  degenerations of Calabi-Yau manifolds. It begins as  Candelas,  Green and H\"{u}bsch discovered that  some nodal  degenerations of Calabi-Yau 3-folds   have
 finite Weil-Petersson distances (cf.  \cite{CGH}).
In general,   \cite{Wang1} proves  that for  a degeneration of Calabi-Yau $n$-manifolds,
  the Weil-Petersson distance between general fibers and the central fiber is finite, if
 the central fiber is  a  Calabi-Yau variety. Wang conjectured     in \cite{Wang2} that the converse is also true, and proposed  to attack the conjecture by using the minimal model program. Eventually, 
   \cite{To,Ta} prove Wang's conjecture, and  obtain  the equivalency between the finiteness of Weil-Petersson metric and the central filling-in with Calabi-Yau varieties under various conditions.  In \cite{Ta},   the further equivalence to the properties of  Ricci-flat K\"{a}hler-Einstein  metrics  is also  established. The equivalencies are used in \cite{zha1} to construct completions for moduli spaces for  polarized  Calabi-Yau manifolds.
   The goal of this note is to  add two more perspectives  to  this picture.

   If $(\pi: \mathcal{X}\rightarrow \Delta, \mathcal{L})$ is  a  degeneration of polarized Calabi-Yau manifolds, then  all fibers $X_{t}$, $t\in \Delta^*$, are diffeomorphic to each other, and we denote
  $X$  the underlying differential manifold.
    Let $\mathbb{H}\rightarrow \Delta^*$ be  the universal covering given by $w\mapsto t=\exp 2\pi\sqrt{-1} w$, where $\mathbb{H}=\{w\in \mathbb{C}| {\rm Im} (w)>0\} $,  and $\tilde{\pi}: \tilde{\mathcal{X}}\rightarrow \mathbb{H}$ be the pull-back family of $\mathcal{X}_{\Delta^*}=\pi^{-1}(\Delta^*) \rightarrow \Delta^*$. The total space  $\tilde{\mathcal{X}}$ is diffeomorphic to $X\times \mathbb{H}$, and we identify $H^{n}(X_{t}, \mathbb{C})$, $t\neq 0$, canonically with $H^{n}(X, \mathbb{C})$. Under  this setup,  we have the first  theorem of this paper.

    \begin{theorem}\label{thm1}  Let $(\pi: \mathcal{X}\rightarrow \Delta, \mathcal{L})$ be  a  degeneration  of polarized Calabi-Yau manifolds.
  Let $\Omega_t$ be  the holomorphic volume form on $X_t$, $t\in\Delta^{*}$,  i.e. a nowhere vanishing section of the canonical bundle $\mathcal{K}_{X_t}$,   such that \begin{equation}\label{equ0.1} (-1)^{\frac{n^2}{2}}\int_{X_t}\Omega_t\wedge \overline{\Omega}_t\equiv 1,\end{equation} and $\omega_{t}$ be  the unique Ricci-flat  K\"{a}hler-Einstein   metric
in  $
c_{1}(\mathcal{L})|_{X_{t}}\in H^{1,1}(X_{t}, \mathbb{R})$.
  Then the  following statements are equivalent.
 \begin{itemize}
  \item[i)] When $t\rightarrow 0$,  the cohomology classes    $$ [ \Omega_t] \rightarrow \beta   \  \  \  \ {\it  in } \  \  \ H^{n}(X, \mathbb{C}).$$
   \item[ii)]   When $t\rightarrow 0$,     $$(X_{t}, \omega_{t})
\stackrel{d_{GH}}\longrightarrow (Y, d_{Y}),  $$ in the Gromov-Hausdorff sense, where    $(Y, d_{Y}) $ is a compact metric space.
\item[iii)]  The origin $0\in \Delta$ is at finite  Weil-Petersson distance from $\Delta^{*}$. \end{itemize}
 \end{theorem}

  The purpose of this theorem  is to present the equivalences obtained in \cite{Wang1,Wang2,To,Ta} from a more Riemannian geometric point of view.  We regards $X_t$, $t\neq 0$,  as $X$ equipped with a complex structure $J_t$, and the metric $ \omega_{t}$ as a Riemannian metric  with  holonomy group   $SU(n)$.  The holomorphic volume form
      $\Omega_t$  is parallel with respect to any Ricci-flat K\"{a}hler-Einstein metric on $X_t$, and after a certain normalization, $ \Omega_{t}$  gives two  calibration  $n$-forms  ${\rm Re}(\Omega_t)$ and ${\rm Im}(\Omega_t)$ in the sense of \cite{HL}.
      Therefore, Theorem \ref{thm1} gives a criterion of Gromov-Hausdorff convergence of Ricci-flat K\"{a}hler-Einstein metrics via the cohomology  classes of calibration  forms in the context of special holonomy  Riemannian geometry. It is desirable to remove the algebro-geometric conditions, for example $X_t$ fitting into an algebraic family with respect to the parameter $t$, and to prove it directly  without using the sophisticated algebraic geometry and PDE.

   Now we switch to a more analytic point of view. We recall that a degeneration of polarized Calabi-Yau manifolds  $(\pi: \mathcal{X}\rightarrow \Delta, \mathcal{L})$ is called having only  log-canonical singularities, if $\mathcal{X}$ is normal, $X_0$ is reduced,  $\mathcal{K}_{\mathcal{X}}+X_{0}$ is $\mathbb{Q}$-Cartier, and for any log-resolution $\bar{\pi}: \bar{\mathcal{X}}\rightarrow \mathcal{X}$ of singularities,
   $$\mathcal{K}_{\bar{\mathcal{X}}}+ X_{0}' \cong_{\mathbb{Q}}  \bar{\pi}^{*}(\mathcal{K}_{\mathcal{X}}+X_{0})+\sum\limits_{E} a_{E}E, \ \ {\rm and} \  \  a_{E} \geq 0, $$ where $E$ are effective  exceptional prime divisors, and  $X_{0}'$ is the strict transform of $X_{0}$ (cf. \cite{KM}).

 Let $(\pi: \mathcal{X}\rightarrow \Delta, \mathcal{L})$ be a  degeneration of polarized Calabi-Yau $n$-manifolds.  There is an integer  $m\geq 1$ such that $\mathcal{L}^{m}$ is relative  very ample, which induces a relative  embedding $\Phi: \mathcal{X} \hookrightarrow \mathbb{CP}^{N}\times\Delta$ such that $\mathcal{L}^{m}\cong \Phi^{*} \mathcal{O}_{\mathbb{CP}^{N}}(1)$, and the restriction $\Phi_{t}=\Phi|_{X_{t}}$ embeds $X_{t}$ into $\mathbb{CP}^{N}$ for any $t\in\Delta$. Note that the choice of  $\Phi$ is not unique.
  If  $\omega_{t}$ is   the unique Ricci-flat  K\"{a}hler-Einstein   metric
in  $
c_{1}(\mathcal{L}^{m})|_{X_{t}}\in H^{1,1}(X_{t}, \mathbb{R})$, and $\omega_{FS}$ denotes the Fubini-Study metric on $\mathbb{CP}^{N}$, then for any $t\in \Delta^{*}$,  there is a unique function
 $\varphi_{t}$, called  the  potential function,  on $X_t$  such that $$\omega_{t}=\Phi_{t}^{*}\omega_{FS}+\sqrt{-1}\partial\overline{\partial}\varphi_{t}, \  \  \  \  \  \sup_{X_{t}} \varphi_{t}=0.$$  The second  theorem is the following.

  \begin{theorem}\label{thm2} Let $(\pi: \mathcal{X}\rightarrow \Delta, \mathcal{L})$ be a  degeneration of polarized Calabi-Yau $n$-manifolds such that $\mathcal{X}$ is normal, the relative canonical bundle $\mathcal{K}_{\mathcal{X}/\Delta}$ is trivial, i.e. $\mathcal{K}_{\mathcal{X}/\Delta}\cong \mathcal{O}_{\mathcal{X}}$, and $\mathcal{X}\rightarrow \Delta$ has at worst log-canonical singularities.
     Let  $\omega_{t}$ be  the unique Ricci-flat  K\"{a}hler-Einstein   metric
presenting  $
c_{1}(\mathcal{L}^{m})|_{X_{t}}\in H^{1,1}(X_{t}, \mathbb{R})$, $t\in\Delta^{*}$, where $m\geq  1$ such that $ \mathcal{L}^{m}$ is relative very ample.
 Then the  following statements are equivalent.
 \begin{itemize}
  \item[i)] There is a relative embedding $\Phi: \mathcal{X} \hookrightarrow \mathbb{CP}^{N}\times\Delta$ induced by $ \mathcal{L}^{m}$ such that the potential functions  $\varphi_{t}$ determined    by $\Phi$ and  $\omega_{t}$ satisfy
      $$ \inf_{X_{t}} \varphi_{t} \geq -C,$$ for a
   constant $C>0$ independent of $t$.
   \item[ii)] The central fiber  $X_0$ is a Calabi-Yau variety.
    \item[iii)]  When $t\rightarrow 0$,    $$(X_{t}, \omega_{t})
\stackrel{d_{GH}}\longrightarrow (Y, d_{Y}),  $$ in the Gromov-Hausdorff sense, where    $(Y, d_{Y}) $ is a compact metric space.  \end{itemize}
 \end{theorem}

 Note that in the condition i) of Theorem \ref{thm2}, if we choose a different  embedding $\Phi'$, then the new potential function $\varphi_{t}'=\varphi_{t}+\xi $, where $\xi$ is a continuous  function on $\mathcal{X}$ such that $\Phi'^{*}\omega_{FS}=\Phi^{*}\omega_{FS}-\sqrt{-1}\partial\overline{\partial}\xi$ on any $X_t$.  Hence the uniformly boundedness   of $\varphi_{t}$ is  equivalent to the boundedness of $\varphi_{t}'$, and the precise bound is not essential.   We can replace  i) by saying that for any embedding induced by $ \mathcal{L}^{m}$, the boundedness of potential functions holds.
  This theorem shows the equivalency between such  boundedness condition   and many other equivalent properties studied in \cite{Wang1,To,Ta}.

  Section 2 gives an exposition of a result due to C.-L. Wang, which shows the equivalency between the convergence of cohomology classes and the finiteness of Weil-Petersson distances.
 We prove Theorem \ref{thm1} in Section 3 and Theorem \ref{thm2} in Section 4.

\noindent {\bf Acknowledgements:}
 The author would  like to thank Valentino Tosatti for inviting him  to write Theorem \ref{thm2} and many comments, and thank   Xiaowei Wang for helpful discussions. This paper was written when the author worked at the Imperial College London. He thanks the mathematics department for the nice working environment.

\section{ Hodge theoretic criterion  }
 Let $(\pi: \mathcal{X}\rightarrow \Delta, \mathcal{L})$ be a  degeneration of polarized Calabi-Yau $n$-manifolds, and
$(\pi: \mathcal{X}^*=\pi^{-1}(\Delta^*) \rightarrow \Delta^*, \mathcal{L})$ be the base change by the inclusion   $\Delta^*\hookrightarrow \Delta$.
A natural  K\"{a}hler metric, possibly  degenerated, is defined on $\Delta^*$, called the  Weil-Petersson metric, which measures the deformation of complex structures of fibers $X_{t}$.  If $\Theta_t$ is a relative holomorphic volume form, i.e. a  no-where vanishing  section of the relative canonical bundle $ \mathcal{K}_{\mathcal{X}^*/\Delta^*}$, then
 the Weil-Petersson metric is $$ \omega_{WP}=-\frac{\sqrt{-1}}{2\pi}\partial\overline{\partial}\log \int_{X_{t}}(-1)^{\frac{n^{2}}{2}} \Theta_{t} \wedge \overline{\Theta}_{t}\geq0,$$   (cf. \cite{Ti,Tod}).    The Weil-Petersson metric $ \omega_{WP}$  is the curvature of the first Hodge bundle  $\pi_{*}\mathcal{K}_{\mathcal{X}^*/\Delta^*}$  with a natural  Hermitian metric.

We rephrase   Corollary  1.2 in \cite{Wang1} as the following.

  \begin{proposition}\label{prop1} Let $X$ be the underlying differential manifold of general  fibers $X_{t}$, $t\in \Delta^*$, and let    $\Omega_t$ be  the holomorphic volume form on $X_t$, $t\in\Delta^{*}$,  such that $$(-1)^{\frac{n^2}{2}}\int_{X_t}\Omega_t\wedge \overline{\Omega}_t\equiv 1.$$
 Then the  following statements are equivalent.
 \begin{itemize}
  \item[i)] When $t\rightarrow 0$,     $$ [ \Omega_t] \rightarrow \beta   \  \  \  \ {\it  in } \  \  \ H^{n}(X, \mathbb{C}).$$
    \item[ii)]  The origin $0\in \Delta$ is at finite  Weil-Petersson distance from $\Delta^{*}$.   \end{itemize}
 \end{proposition}

 \begin{remark} Note that $\Omega_t$ does not vary holomorphically with respect to the variable $t$  due to the normalization condition, and it is only a smooth section of the relative canonical bundle.
 \end{remark}

 \begin{proof}
We recall  the argument in Section 1 of \cite{Wang1} with only
  a minor modification, which relies on the classical theory of Hodge degenerations (See \cite{Gr} and \cite{GHJ}  for the necessary  background knowledge  in the proof).  If $\mathbb{H}\rightarrow \Delta^*$ is the universal covering given by $w\mapsto t=\exp 2\pi\sqrt{-1} w$, and $\tilde{\pi}: \tilde{\mathcal{X}}\rightarrow \mathbb{H}$ is the base change  of $\mathcal{X}^* \rightarrow \Delta^*$, then $\tilde{\mathcal{X}}$ is diffeomorphic to $X\times \mathbb{H}$. We identify $H^{n}(X_{t}, \mathbb{C})$, $t\neq 0$, canonically with $H^{n}(X, \mathbb{C})$, and we always write $\tilde{\pi}^{-1}(w)$ as $X_t$ for $ t=\exp 2\pi\sqrt{-1} w$.  If $\tilde{\mathcal{L}}$ denotes the pull-back of $\mathcal{L}$ to $\tilde{\mathcal{X}}$, then the first Chern class $c_{1}(\tilde{\mathcal{L}})\in H^{2}(X, \mathbb{Z})$. 

 The family of polarized Calabi-Yau $n$-manifolds  $(\pi: \mathcal{X}^*  \rightarrow \Delta^*, \mathcal{L})$ gives  a polarized variation of Hodge structures of weight $n$. The Hodge filtrations are $F^{n}_{t}\subset \cdots \subset F^{0}_{t}=V$, for any $t\in \Delta^{*}$, where $V\subset H^{n}(X, \mathbb{C})$ is the primitive cohomology   with respect to $c_{1}(\tilde{\mathcal{L}})$, and $F^{n}_{t}=H^{n,0}(X_{t})$. The polarization is the Hodge-Riemann bilinear form $$Q(\phi, \psi)=(-1)^{\frac{n(n-1)}{2}}\int_{X} \phi\wedge \psi ,$$ for any $\phi $ and $ \psi \in V$,  which satisfies $$Q(F^{p},F^{n+1-p})=0, \  \  \ \ {\rm and } \  \  \  (-1)^{\frac{p-q}{2}}Q(\xi,\bar{\xi})>0$$ for any $\xi \in V \cap H^{p,q}(X_{t})$.

Let $T: H^{n}(X, \mathbb{Z}) \rightarrow H^{n}(X, \mathbb{Z})$ be the monodromy operator induced by the parallel transport of the local system $R^{n}\pi_{*}\mathbb{Z}$ along the loop generating $\pi_{1}(\Delta^{*})$. Since the polarization is invariant under parallel transports of $R^{2}\pi_{*}\mathbb{Z}$, $T$ acts on $V$.  Note that  $T$ is quasi-unipotent, and therefore, we assume that $T$ is unipotent by passing to a certain base change, i.e. $(T-I)^{k}=0$ for a $k\leq n+1$ (cf. \cite{Gr}).  We let $N=\log T: H^{n}(X, \mathbb{C}) \rightarrow H^{n}(X, \mathbb{C})$, which satisfies $N^{k}=0$ and $Q(N\phi, \psi)=-Q(\phi, N\psi)$ for any $\phi $,  $ \psi \in V$ (cf.  \cite{Gr}).

Since $h^{n,0}(X_t)=1$, we only consider  the period map $\mathcal{P}: \mathbb{H}\rightarrow \mathbb{P}(V)$ of the $n$th-flag, which  is defined by $\mathcal{P}(w)=\langle[\Theta_w]\rangle $, where $\Theta_w$ is a holomorphic volume form on $X_{t}$, and $\langle[\Theta_w]\rangle$ denotes the complex line in $ V$ determined by $[\Theta_w] $. If we define $\tilde{\Xi}: \mathbb{H}\rightarrow \mathbb{P}(V)$ by $\tilde{\Xi}(w)=e^{-w N}\mathcal{P}(w)$, then $\tilde{\Xi}$ descents to a map $\Xi: \Delta^{*} \rightarrow \mathbb{P}(V)$.  The Schmid's nilpotent orbit theorem asserts that $\Xi$  extends  to a holomorphic map from $\Delta$, denoted still by $\Xi:  \Delta  \rightarrow \mathbb{P}(V)$ (cf.  \cite{Sc}, Chapter IV in \cite{Gr},  and Section 16.3 in \cite{GHJ}).

 If $A: \Delta \rightarrow V$ is  a holomorphic map such that $A(t)\neq 0$ for any $t\in \Delta$, and $\langle A(t)\rangle \in \Xi(t)$, then $\langle e^{wN}A(t)\rangle\in \mathcal{P}(w)$, 
  for any $w\in \mathbb{H}$ with $t=\exp 2\pi\sqrt{-1}w$. Let $\Theta_{t}$ be a relative   holomorphic volume form on $\mathcal{X}^{*}$   such that
   $[\Theta_{t}]=e^{wN}A(t)$ on $X_{t}$.   Note that $A(t)=a_0+h(t)$ where $a_0=A(0)$ and $h(t)$ is a holomorphic vector valued  function on $\Delta$ with $|h(t)|\leq C|t|$ for a constant $C>0$. Here $|\cdot|$ denotes  a fixed   Euclidean norm  on the finite dimensional vector space $V$.    Therefore \begin{equation}\label{equ1.1} [\Theta_{t}]= a_0+\sum_{j=1}^{d}\frac{w^{j}}{j!}N^{j}a_0+e^{wN}h(t), \end{equation} where $d=\max \{l| N^{l}a_0\neq 0\}$.

 Since $N$ is anti-symmetric with respect to $Q$, we have  $Q(\cdot, e^{\bar{w}N} \cdot)=Q( e^{-\bar{w}N}\cdot,\cdot)$, and
  $$Q([\Theta_{t}],[\overline{\Theta}_{t}])  =  Q(e^{2\sqrt{-1}yN}A(t),\overline{A(t)})  = Q(e^{2\sqrt{-1}yN}a_0,\overline{a_0})+b(t) =p(y)+b(t),  $$ where $y={\rm Im}(w)=\frac{-\log|t|}{2\pi}$, and  $p(y)$ is a polynomial of $y$ with degree less or equal to $d$. Since $N^{n+1}h(t)=0$,
   $|b(t)|\leq C|t|(-\log |t|)^{n} $ for a certain constant $C>0$.
   In Section 1 of \cite{Wang1}, a detailed study of  the mixed Hodge structure  shows  that $Q(N^{d}a_0,\bar{a}_0)\neq 0$, and therefore the degree of   $p(y)$ is  $d$.   We obtain that  \[\begin{split}\sigma_t& =(-1)^{\frac{n^2}{2}}\int_{X_t}\Theta_t\wedge \overline{\Theta}_t =(-1)^{\frac{n}{2}}Q([\Theta_{t}],[\overline{\Theta}_{t}]) \\ & =(-1)^{\frac{n}{2}}p(-\log |t|)+(-1)^{\frac{n}{2}}b(t)\\ & \sim  (-\log |t|)^{d}+o(1). \end{split}\]

We define  $\Omega_{t}=\sigma_t^{-\frac{1}{2}} \Theta_t$, which satisfies
  $(-1)^{\frac{n^2}{2}}\int_{X_t}\Omega_t\wedge \overline{\Omega}_t=1$ and  $$ [\Omega_{t}]=\sigma_t^{-\frac{1}{2}} a_0+\sigma_t^{-\frac{1}{2}}\sum_{j=1}^{d}\frac{w^{j}}{j!}N^{j}a_0+\sigma_t^{-\frac{1}{2}}e^{wN}h(t) $$ by   (\ref{equ1.1}).
  If $d>0$,  then   $$ \sigma_t^{-\frac{1}{2}} |w^{d}N^{d}a_0|\sim  (-\log |t|)^{\frac{d}{2}}, \  \ \  \sigma_t^{-\frac{1}{2}} |w^{j}N^{j}a_0|\sim  (-\log |t|)^{j-\frac{d}{2}}, \  \  $$ $$ 0\leq j<d,   \  \  \  \ {\rm and} \  \ \ \  \sigma_t^{-\frac{1}{2}}|e^{wN}h(t)|\leq C(-\log |t|)^{n-\frac{d}{2}}|t|.$$ Thus $[\Omega_{t}]$ diverges   in $H^{n}(X,\mathbb{C})$.   If $d=0$, then when $t\rightarrow 0 $,
    $$\sigma_{t} \rightarrow  (-1)^{\frac{n^2}{2}} Q(a_0,\bar{a}_0)=\sigma_{0},  \ \  \ {\rm and}  \  \  \     [\Omega_{t}] \rightarrow  \sigma_{0}^{-\frac{1}{2}} a_0, \  \  \ {\rm in} \  \  \  H^{n}(X,\mathbb{C}). $$
    The calculation in Section 1 of \cite{Wang1} shows that   $d=0$ is equivalent to the finiteness of Weil-Petersson distance from the  origin $0$ to  $\Delta^{*}$.
   We obtain the conclusion.
\end{proof}

\section{Proof of Theorem \ref{thm1}}

Before we prove Theorem \ref{thm1}, we show  a result about the Gromov-Hausdorff convergence, which does not appear  explicitly in the literature.

  \begin{theorem}\label{prop2} Let $(\pi: \mathcal{X}\rightarrow \Delta, \mathcal{L})$ be  a  degeneration  of polarized Calabi-Yau manifolds,
  and $\omega_{t}$ be  the unique Ricci-flat  K\"{a}hler-Einstein   metric
in  $
c_{1}(\mathcal{L})|_{X_{t}}\in H^{1,1}(X_{t}, \mathbb{R})$. If the origin $0\in \Delta$ is at finite  Weil-Petersson distance from $\Delta^{*}$,
  then  when $t\rightarrow 0$,  $$(X_{t}, \omega_{t})
\stackrel{d_{GH}}\longrightarrow (Y, d_{Y}),  $$ in the Gromov-Hausdorff sense,  where    $(Y, d_{Y}) $ is a compact metric space, and is  homeomorphic to a Calabi-Yau variety.
 \end{theorem}
 
 This theorem could  be proved by the argument  in the proof of  Lemma 6.9 in  \cite{LWX}, and we provide an  independent proof here. 

\begin{proof} By Corollary 1.7 in \cite{Ta},
 the diameters $${\rm diam}_{\omega_t}(X_t)\leq D,$$ for a constant $D>0$ indpendent of $t$.
 The  Gromov precompactness theorem (cf. \cite{Gr,Fu}) asserts that for any sequence $t_k\rightarrow 0$,  $$(X_{t_{k}}, \omega_{t_k})
\stackrel{d_{GH}}\longrightarrow (Y, d_{Y}),  $$ by passing to a subsequence, in the Gromov-Hausdorff sense, where $Y$
 is a compact metric space.  Now we improve  the convergence to along $t\rightarrow 0$, i.e. without passing to any subsequences.  We follow the arguments in the proof of Theorem 1.4 in  \cite{Ta}.

Let $P$ be the   Hilbert polynomial of the general fibers, i.e.  $P=P(\mu)=\chi (X_t,L_t^{\mu})$ where $L_t=\mathcal{L}|_{X_t}$.
By Matsusaka's Big Theorem (cf. \cite{Mat}),    there is an  $m_{0}>0$ depending only on $P$    such that  for any $m\geqslant m_{0}$,  $L_t^{m}$ is very ample, and $H^{i}(X_t,L_t^{m})=\{0\}$, $i>0$,  $t\in\Delta^{*}$.
 A basis  $\Sigma$ of $H^{0}(X_t,L_t^{m})$ induces  an embedding $\Phi_{t}: X_{t} 
 \hookrightarrow \mathbb{CP}^{N}$ such that  $L_t^{m}=\Phi_{t}^{*}\mathcal{O}_{\mathbb{CP}^{N}}(1)$. We   regard $\Phi_{t}(X_t)$ as a   point in the Hilbert scheme  $\mathcal{H}il_{N}^{P_{m}}$ parametrizing the subshemes of $\mathbb{CP}^{N}$ with   Hilbert polynomial $P_{m}(\mu)=\chi (X_t,L_t^{m \mu})$, where $N=P_{m}(1)-1$.  For any other choice $\Sigma'$,   $\Phi_{t}'(X_t)=\varrho (u, \Phi_{t}(X_t))$ for a $u\in SL(N+1)$ where $\varrho : SL(N+1) \times \mathcal{H}il_{N}^{P_{m}}  \rightarrow  \mathcal{H}il_{N}^{P_{m}}$ is the  $SL(N+1)$-action on $\mathcal{H}il_{N}^{P_{m}}$ induced by the natural $SL(N+1)$-action on $\mathbb{CP}^{N}$ (See \cite{Vie} for the background knowledge). We choose  $m \gg 1$ such that $\mathcal{L}^{m}$ is relative ample on $\mathcal{X}^{*}=\mathcal{X}\backslash X_0$.  

 Theorem 1.2 of \cite{DS}   asserts that by taking  $m\gg 1$, we have a subsequence $X_{t_{k}}$ satisfying  the following.  For any $t_k$,  there is an orthonormal basis $\Sigma_{k}$ of $H^{0}(X_{t_k}, L_{t_k}^{m})$ with respect to the $L^{2}$-norm induced by the Hermitian metric $H_k$ on $L_{t_k}$  giving
  $\omega_{t_k}$, i.e. $\omega_{t_k}=-\sqrt{-1}\partial \overline{\partial}\log H_k$,   which defines an embedding $\Phi_{t_{k}}: X_{t_k} \hookrightarrow \mathbb{CP}^{N}$ with $L_{t_k}^{m}= \Phi_{t_{k}}^{*} \mathcal{O}_{\mathbb{CP}^{N}}(1)$. And  $\Phi_{t_{k}}( X_{t_k})$ converges to $X_{\infty}$ in the reduced  Hilbert scheme $(\mathcal{H}il_{N}^{P_{m}})_{red}$ with respect to the natural analytic topology. Furthermore $X_{\infty}$ is homeomorphic to the Gromov-Hausdorff limit $Y$.
By   Proposition 4.15 of  \cite{DS},   $X_{\infty}$ is  a projective normal variety with only log-terminal singularities.  Note that we can choose  holomorphic volume forms $\Omega_{t_k}$  converging   to a holomorphic volume form $\Omega_{\infty}$ on the regular locus $X_{\infty,reg}$ along the Gromov-Hausdorff convergence. Thus the canonical  sheaf  $\mathcal{K}_{X_{\infty}}$ is trivial, i.e.  $\mathcal{K}_{X_{\infty}}\cong \mathcal{O}_{X_{\infty}}$, and $X_{\infty}$ is 1--Gorenstein, which implies that $X_{\infty}$ has at worst  canonical singularities.    Consequently,  $X_{\infty}$ is  a Calabi-Yau variety.

For any $p \in \mathcal{H}il_{N}^{P_{m}}$, we denote $O_p$ the $SL(N+1)$-orbit, i.e. $O_p=\{\varrho (u, p)| \forall u \in SL(N+1)\}$, and $\overline{O}_p$ the Zariski closure of $O_p$ in $ \mathcal{H}il_{N}^{P_{m}}$.  If $\mathcal{H}^{o}\subset  \mathcal{H}il_{N}^{P_{m}}$ denotes  the open subscheme parameterizing smooth projective manifolds with  Hilbert polynomial $ P_{m}$, then $\overline{O}_p \cap \mathcal{H}^{o}$ is clearly  closed in $\mathcal{H}^{o}$, which works as the following.

 Let $q \in \overline{O}_p \cap \mathcal{H}^{o}$, and $\iota : \Delta \rightarrow \overline{O}_p$ such that $q=\iota(0)$ and $\iota (\Delta^{*})\subset O_p $.  We obtain a family of Calabi-Yau manifolds $\mathcal{Z}\rightarrow \Delta$ as the base change, i.e. $\mathcal{Z}=\mathcal{U}^{P_m}\times_{\mathcal{H}il_{N}^{P_{m}}}\Delta$, where $\mathcal{U}^{P_m} \rightarrow \mathcal{H}il_{N}^{P_{m}}$ denotes the universal family. Note that all fibers $Z_{z}$, $z\in \Delta^{*}$, are isomorphic to each other as $\iota (z)\in O_p$ for any $z\neq 0$.  Thus the image of the period map  $\mathcal{P}:\Delta \rightarrow \mathcal{D}$ is one point,  where $\mathcal{D}$ denotes the  classifying space for the polarized Hodge structure of weight $n$ (cf. \cite{Gr}). The differential of the period map
 $d \mathcal{P}_{z}: T_{z} \Delta \to T_{\mathcal{P}(z)}(\mathcal{D})$
 is a composition of the Kodaira-Spencer map $T_{z}\Delta  \to H^{n-1,1}(Z_{z})$
 and a   map $\eta: H^{n-1,1}(Z_{z}) \to T_{\mathcal{P}(z)}( \mathcal{D})$ (cf. Chapter III in \cite{Gr} and Section 16.2 in \cite{GHJ}). The local  Torelli theorem for Calabi-Yau manifolds says that $\eta$ is injective
  (cf. Proposition 3.6 in  \cite{Gr2}), and thus the Kodaira-Spencer map  of $\mathcal{Z}\rightarrow \Delta$ is trivial. Therefore all fibers $Z_{z}$, $z\in \Delta$, are biholomorphic to each other, denoted by $Z$.  Since $Z$ is simply connected, any restriction $\mathcal{O}_{\mathbb{CP}^{N}}(1)|_{Z_z}$ is isomorphic to  the same ample line bundle $L_Z$, and any $Z_z\subset \mathbb{CP}^{N}$ is the image of the embedding given by a basis of $H^0(Z, L_Z)$.  Hence $q=\iota(0) \in O_p$, and $\overline{O}_p \cap \mathcal{H}^{o}=O_p \cap \mathcal{H}^{o}$.

  Now we continue the proof. By the universal property of the universal  family $\mathcal{U}^{P_m} \rightarrow \mathcal{H}il_{N}^{P_{m}}$, we have a morphism $\lambda: \Delta \rightarrow \mathcal{H}il_{N}^{P_{m}}$ such that $\pi: \mathcal{X}\rightarrow \Delta$ is the pull-back family of $\mathcal{U}^{P_m}$, i.e. the base change $\mathcal{X}=\mathcal{U}^{P_m}\times_{\mathcal{H}il_{N}^{P_{m}}}\Delta$.   The Zariski closure of orbits 
  $$\mathcal{A}=\overline{\bigcup_{t\in \Delta}\{t\}\times O_{\lambda (t)}}\subset \Delta \times \mathcal{H}il_{N}^{P_{m}}$$ is studied in the  part 4) of  the proof of Theorem 1.4 in \cite{Ta}. It is proved in \cite{Ta} that $\mathcal{A}$ is a  irreducible and projective variety  over $\Delta$,  and if $\mathfrak{a}: \mathcal{A} \rightarrow \Delta$ is the restriction of the natural projection,  the  fiber $\mathfrak{a}^{-1}(t)=\{t\}\times \overline{O}_{\lambda (t)}$ for any $t\in\Delta^{*}$. However the central fiber $\mathfrak{a}^{-1}(0)$ may  be reducible.  If we let $$\mathfrak{a}^{o}=\mathfrak{a}|_{\mathcal{A}^{o}}: \mathcal{A}^{o}=\mathcal{A}\cap (\Delta^{*}\times \mathcal{H}^{o})\rightarrow\Delta^{*}, $$ then $\mathcal{A}^{o}$ is a Zariski open set of $\mathcal{A}$, and $$\mathfrak{a}^{o,-1}(t)=\{t\}\times (\overline{O}_{\lambda (t)}\cap  \mathcal{H}^{o})=\{t\}\times ( O_{\lambda (t)}\cap  \mathcal{H}^{o}),  \  \  \  t\in\Delta^{*}. $$

  Note that $(t_k,\Phi_{t_{k}}( X_{t_k}))\in \mathcal{A}^{o}$ and $(0, X_\infty) \in \mathfrak{a}^{-1}(0) \subset  \mathcal{A}$. But $(0, X_\infty)$ may not belong to $ O_{\lambda (0)}$.     By the part 5) in the proof of Theorem 1.4 in \cite{Ta},  we find a morphism $\nu : \Delta \rightarrow \mathcal{A}$ such that $\nu (0)=(0, X_\infty)$, $ \nu (\Delta^{*})\subset \mathcal{A}^{o}$,  and the composition $\mathfrak{a}\circ \nu$ is a finite map  given by $s \mapsto t=s^{l}$ by shrinking $\Delta$ if necessary.  We denote  the pull-back family $$\pi': \mathcal{X}'= (\mathcal{U}^{P_m}\times_{\mathcal{H}il_{N}^{P_{m}}}\Delta)_{red}\rightarrow \Delta$$ by $\mathfrak{p}\circ \nu$, where $\mathfrak{p}: \mathcal{A} \rightarrow \mathcal{H}il_{N}^{P_{m}}$ is the restriction of the natural projection map.  
Let  $\mathcal{L}'$ be the pull-back bundle of $\mathfrak{O}_{\mathbb{CP}^{N}}(1)$, which is a relative very  ample line bundle on $ \mathcal{X}'$, where $\mathfrak{O}_{\mathbb{CP}^{N}}(1)$ is the line bundle on $\mathcal{U}^{P_m}$ induced by  $\mathcal{O}_{\mathbb{CP}^{N}}(1)$.  Note that the central fiber $X_{0}'=\pi'^{-1}(0)=X_\infty$, and for any $t=s^l$,  $X_{s}'$ is isomorphic to $X_t$, since  $\nu(s)$  belongs  to $\mathfrak{a}^{o,-1}(t)=\{t\}\times O_{\lambda (t)}$.  More explicitly, the isomorphism is given by an element  $u_s \in SL(N+1)$ such that $\varrho (u_s, \mathfrak{p}\circ \nu (s) )= \lambda (t)$. The restricted bundle $\mathcal{L}'|_{X_s'}\cong L_t^m$.

    We have a new polarized degeneration of Calabi-Yau manifolds $(\pi': \mathcal{X}'\rightarrow \Delta, \mathcal{L}')$ with  a Calabi-Yau variety $X_\infty$ as the central fiber.  Since   $X_{\infty}$ is  normal,    the total space $\mathcal{X}'$ is normal, and thus
      the relative canonical  sheaf $\mathcal{K}_{\mathcal{X}'/ \Delta}$  is defined, i.e. $\mathcal{K}_{\mathcal{X}'/ \Delta}\cong \mathcal{K}_{\mathcal{X}'}\otimes \pi'^{*}\mathcal{K}_{\Delta}^{-1}$,  and is  trivial, i.e. $\mathcal{K}_{\mathcal{X}'/ \Delta}\cong \mathcal{O}_{\mathcal{X}'}$.  If $\omega_s'$ is the unique Ricci-flat K\"{a}hler-Einstein metric presenting $c_{1}(\mathcal{L}'|_{X_s'})$, then $\omega_s'=m\omega_t$ after we identify $X_s'$ and $X_t$ with $t=s^l$.
      In this case, the convergence of $\omega_{s}' $ is studied in \cite{RuZ,RoZ1,RoZ2}.
 It is proved in \cite{RoZ1} that $$  F_{s}^{*}\omega_{s}' \rightarrow  \omega, \  \ \ {\rm
when} \ \
 s \rightarrow 0,$$
 in the $C^{\infty}$-sense on  any compact subset $K $ belonging to the regular part $
X_{0,reg}'$ of $X_0'$,  where $F_{s}: X_{0,reg}' \rightarrow
X_{s}'$ is a smooth  family of embeddings with $F_{0}={\rm Id}_{X_{0}'}$, and  $\omega$ is   a  Ricci-flat K\"{a}hler-Einstein metric   on $X_{0,reg}'$ with    $\omega\in
c_{1}(\mathcal{L'})|_{X_{0}'}$, which was obtained previously in \cite{EGZ}.  Furthermore, \cite{RoZ2} proves  $$(X_{s}', \omega_{s}')
\stackrel{d_{GH}}\longrightarrow (Y', d_{Y'}),  $$ when $s\rightarrow 0$, in the Gromov-Hausdorff sense, where    $(Y', d_{Y'}) $   is the metric completion of $(X_{0,reg}', \omega)$, which
 is a compact metric space.   Note that $t\rightarrow 0$ if and only if $s\rightarrow 0$, and $Y$ is homeomorphic to $Y'$, and $d_{Y'}=\sqrt{m}d_Y$. We obtain the conclusion.
      \end{proof}

      \begin{remark} One crucial  step  in the proof is to replace the  original  degeneration  $(\pi: \mathcal{X}\rightarrow \Delta, \mathcal{L})$  by   a new one $(\pi': \mathcal{X}'\rightarrow \Delta, \mathcal{L}')$, which satisfies that  $\mathcal{X}'$ contains all smooth fibers of $\mathcal{X}$, and the new central fiber $X_0'$ is a Calabi-Yau variety. Furthermore,  under the  identification of two general fibers   $X_t \cong X_s'$, $t=s^l$, we have $ \mathcal{L}^{m}|_{X_t}\cong \mathcal{L}'|_{X_s'}$. By Corollary 2.3 in \cite{zha1}, the Calabi-Yau variety $X_0'$ is the unique choice as the filling-in in the following sense. If $(\pi'': \mathcal{X}''\rightarrow \Delta, \mathcal{L}'')$ is an another   degeneration with the  Calabi-Yau  variety    $X_{0}''$ as the central fiber, and if  there is a sequence of points  $s_{k} \rightarrow 0$ in $\Delta$ such that  $ X_{s_{k}}'\cong  X_{s_{k}}''$ and  $\mathcal{L}'|_{X'_{s_{k}}}\cong  \mathcal{L}''|_{X_{s_{k}}''}$, then $X_{0}'$ is isomorphic to $X_{0}''$ (see also  \cite{Bo,LWX,Od,SSY}).  Such property is called the separatedness  condition, and is used to construct certain completions of moduli spaces (\cite{LWX,SSY,zha1,Zh}).

     The other way to find a Calabi-Yau variety as the  filling-in is to use the minimal model program as proposed by Wang (\cite{Wang2}), and carried out in \cite{To,Ta}.   In this case,
      we further assume that   the degeneration $ (\mathcal{X}\rightarrow \Delta, \mathcal{L})$ comes from a quasi-projective family.
  More explicitly, there is a flat family of polarized $n$-varieties $( \mathfrak{X}\rightarrow C, \mathfrak{L})$ over a smooth curve $C$ with a marked point $y$ such that both of $\mathfrak{X}$ and $C$ are quasi-projective,
  $\mathcal{X}$ is the pull-back family  of $\mathfrak{X}$ for  an embedding $ \Delta \hookrightarrow C$ mapping $0$ to  $y$,  and $\mathcal{L}$ is the pull-back bundle of $\mathfrak{L}$.  All examples the author know satisfy this assumption.

  Now we follow the arguments in \cite{To,Ta}, and by taking the  Mumford's semi-stable reduction, i.e. a sequence of base changes and blow-ups, we  obtain a degeneration $\tilde{\pi}: \tilde{\mathcal{X}}\rightarrow \Delta$ with a normal crossing central fiber $\tilde{X}_0$ and $\tilde{\mathcal{X}}\backslash \tilde{X}_{0}$ being a base change of $\mathcal{X}\backslash X_0$ by $s \mapsto s^j=t$.
   Then  \cite{To,Ta} use the recent results in the minimal model program, for example \cite{Fuj,NX}, to show that   $\tilde{\mathcal{X}}\rightarrow \Delta$ is birational to a family $\tilde{\mathcal{X}}'\rightarrow \Delta$ such that $\tilde{\mathcal{X}}\backslash \tilde{X}_{0}\cong \tilde{\mathcal{X}}'\backslash \tilde{X}_{0}' $, and the relative canonical divisor $\mathcal{K}_{\tilde{\mathcal{X}}'/ \Delta}$ is $\mathbb{Q}$-Cartier, and $\mathbb{Q}$-linearly trivial. A further argument proves that $\mathcal{K}_{\tilde{\mathcal{X}}'/ \Delta}$ is Cartier, and linearly trivial (see the proof of Theorem 1.2 in \cite{To} for details).  Thus we have a Calabi-Yau variety $\tilde{X}_{0}' $ as the central filling-in.

It is known that the minimal model $\tilde{\mathcal{X}}'$ is not uniquely chosen, and it is proved in \cite{Kaw} that  any other choice $ \tilde{\mathcal{X}}'' $ connects to $\tilde{\mathcal{X}}'$ by a sequence of flops. Therefore the central Calabi-Yau  variety $\tilde{X}_{0}' $ obtained by the minimal model program  is not unique. Comparing to the unique chosen $X_0'$ in the proof of Theorem \ref{prop2}, what happens is the following. If  $\tilde{\mathcal{L}}$ is  the pull-back bundle of $\mathcal{L}$ on $\tilde{\mathcal{X}}\backslash \tilde{X}_{0}$ and therefore on $ \tilde{\mathcal{X}}'\backslash \tilde{X}_{0}' $, then $\tilde{\mathcal{L}}$ is relative ample. If one minimal model $\tilde{\mathcal{X}}'$ allows that  $\tilde{\mathcal{L}}$ extends to a relative ample line bundle on  $\tilde{\mathcal{X}}'$   crossing the central fiber $\tilde{X}_0'$ after taking a certain power, then $\tilde{\mathcal{X}}'$ is the only minimal model among many possible choices allowing the ample extension of   $\tilde{\mathcal{L}}$ by the separatedness condition (cf.  Theorem 2.1 in  \cite{Bo}). And this uniquely chosen  $\tilde{X}_0'$ would coincide with the Gromov-Hausdorff limit of the Ricci-flat K\"{a}hler-Einstein metrics representing the polarization on the nearby fibers.   \end{remark}

Now we are ready to prove Theorem \ref{thm1}.

  \begin{proof}[Proof of Theorem \ref{thm1}]   Firstly, we show  that i) implies ii).  By  Proposition \ref{prop1},  $0\in \Delta $ is at finite Weil-Petersson distance from $\Delta^{*}$, and thus ii) is the consequence of Theorem  \ref{prop2}. 

  Now if we assume that ii) is true, then clearly there is a diameter upper  bound, i.e. ${\rm diam}(\omega_{t})\leq D$ for a constant $D>0$, which is equivalent to the volume non-collapsing property, i.e. for any $r< {\rm diam}(\omega_{t})$, $${\rm Vol}_{\omega_{t}}(B_{\omega_{t}}(r))\geq \kappa r^{2n},$$ for a constant $\kappa >0$, by the Bishop-Gromov comparison theorem, and the polarization condition $${\rm Vol}_{\omega_{t}}(X_{t})=\frac{1}{n!}\int_{X_t} c_{1}(\mathcal{L})^{n}>0. $$ By
  Theorem 1.4 of \cite{Ta},
    $0\in \Delta $ is at finite Weil-Petersson distance from $\Delta^{*}$.  We obtain i) by Proposition \ref{prop1}.
   \end{proof}

\section{Proof of Theorem \ref{thm2} }

 \begin{proof}[Proof of Theorem \ref{thm2}]
 Firstly, the $C^{0}$-estimate in Section 3 of \cite{RoZ1} shows that ii) implies i).  Now we assume that i) is true, and we denote  $\omega_{t}^{o}=\Phi_{t}^{*}\omega_{FS} \in c_{1}(\mathcal{L}^{m})|_{X_{t}}$ on $X_{t}$, $t\in\Delta^{*}$.

 Let $\Psi_t$ be a nowhere vanishing section of $\mathcal{K}_{\mathcal{X}/\Delta}$, i.e. the divisor ${\rm div} (\Psi_t)=0$. Note that the  codimension of  the singular set $S_{\mathcal{X}}$ of $\mathcal{X}$ is  bigger or equal to 2, since $\mathcal{X}$ is normal, and any irreducible component of $X_0$ has multiplicity one as $X_0$ is reduced.  If $p\in  \mathcal{X}_{reg} \cap X_{0,reg}$, where $\mathcal{X}_{reg}=\mathcal{X}\backslash S_{\mathcal{X}}$ and $X_{0,reg} $ denotes the regular set of $X_0$, then there is a neighborhood $U\subset \mathcal{X}_{reg}$ of $p$ such that $U\cap X_{0,reg} \subset X_{0,reg} $, and there are coordinates $z_{0}, z_{1}, \cdots, z_{n}$ on $U$ satisfying that $X_{t}\cap U=\{z_{0}=t\}$, and   $z_{1}, \cdots, z_{n}$ are coordinates on $X_{t}\cap U$. Therefore, there is a nowhere vanishing holomorphic function $h_{U}$ on $U$ such that $\Psi_t=h_{U}dz_{1}\wedge \cdots \wedge dz_{n}$. Since $\omega_{t}^{o}$ is smooth on $U$, we have
\begin{equation}\label{equ3.1}  (-1)^{\frac{n^2}{2}}\Psi_t \wedge \overline{\Psi}_t \leq C_{U} (\omega_{t}^{o})^{n},
  \end{equation}
 on $X_{t}\cap U'$ for a constant $C_{U}>0$, and a smaller $U' \subset U$.

 The Ricci-flat condition is equivalent to that the potential function $\varphi_t$ satisfies the Monge-Amp\`{e}re equation \begin{equation}\label{equ3.2}\omega_{t}^{n}=(\omega_{t}^{o}+ \sqrt{-1}\partial\overline{\partial}\varphi_{t})^{n}=e^{\rho_t} (-1)^{\frac{n^2}{2}}\Psi_t \wedge \overline{\Psi}_t,  \end{equation} where $\rho_t$ is a constant function when restricted on $X_{t}$.    The argument in Section 3 of \cite{RoZ1} shows a generalized Yau-Schwartz lemma, i.e.  \begin{equation}\label{equ3.3} \omega_{t}^{o}\leq C  \omega_{t}  \end{equation} for a constant $C>0$ independent of $t$.  The proof is as the following.
  If
 $\Phi_{t}: (X_{t}, \omega_{t}) \rightarrow  (\mathbb{CP}^{N},
\omega_{FS})$ is  the inclusion   map induced by
$\mathcal{X}\subset \mathbb{CP}^{N}\times\Delta $, the Chern-Lu
inequality says
$$\Delta_{\omega_{t}}\log |\partial \Phi_{t}|^{2}\geq \frac{{\rm Ric}_{\omega_{t}}(\partial \Phi_{t},
\overline{ \partial \Phi_{t}})}{|\partial \Phi_{t}|^{2}}-\frac{{\rm
Sec}(\partial \Phi_{t}, \overline{
\partial \Phi_{t}},
\partial \Phi_{t},
\overline{ \partial \Phi_{t}})}{|\partial \Phi_{t}|^{2}},
$$  where ${\rm
Sec}$ denotes  the holomorphic bi-sectional curvature of
$\omega_{FS}$ (cf.  \cite{Y3}). Note that
$\Phi_{t}^{*}\omega_{FS}=\omega_{t}^{o} $,  $|\partial
\Phi_{t}|^{2}={\rm
tr}_{\omega_{t}}\omega_{t}^{o}=n-\Delta_{\omega_{t}}\varphi_{t}$
and ${\rm Ric}_{\omega_{t}}=0$. Thus we have that
 $$\Delta_{\omega_{t}}(\log {\rm tr}_{\omega_{t}}\omega_{t}^{o}-2\overline{R}\varphi_{t})
\geq -2\overline{R}n+\overline{R}{\rm
tr}_{\omega_{t}}\omega_{t}^{o}.
$$ where $\overline{R}$ is  a constant  depending only the upper bound of ${\rm
Sec} $. By the  maximum principle, there is an $x\in X_{t}$ such that
${\rm tr}_{\omega_{t}}\omega_{t}^{o}(x)\leq 2n$, and
 $${\rm tr}_{\omega_{t}}\omega_{t}^{o}\leq 2ne^{2\overline{R}(\varphi_{t}-\varphi_{t}(x))}\leq C,  $$ by the assumption  i),
 where $C>0$ is  a constant  independent of $t$, and we obtain (\ref{equ3.3}).

By (\ref{equ3.2})  and (\ref{equ3.3}),  we have $$ C^{-n}( \omega_{t}^{o})^{n}  \leq  \omega_{t}^{n}=e^{\rho_t} (-1)^{\frac{n^2}{2}}\Psi_t \wedge \overline{\Psi}_t ,$$ and after we restrict this inequality on $U$, (\ref{equ3.1}) asserts that $$e^{\rho_t}\geq C_{1}$$ for a constant $C_{1}$ independent of $t$.  We have $$C_{1} \int_{X_{t}} (-1)^{\frac{n^2}{2}}\Psi_t \wedge \overline{\Psi}_t\leq e^{\rho_t} \int_{X_{t}} (-1)^{\frac{n^2}{2}}\Psi_t \wedge \overline{\Psi}_t =  \int_{X_{t}} \omega_{t}^{n}= \int_{X_{t}}c_{1}(\mathcal{L}^{m})^{n},   $$ and  we obtain  ii) by Corollary 1.5 of \cite{Ta}.

Under the assumption ii),   \cite{RoZ2} proves  $$(X_{t}, \omega_{t})
\stackrel{d_{GH}}\longrightarrow (Y, d_{Y}),  $$ when $t\rightarrow 0$, in the Gromov-Hausdorff sense, where    $(Y, d_{Y}) $   is the metric completion of $(X_{0,reg}, \omega)$, which
 is a compact metric space (See also the survey paper \cite{Zh} for more discussions  of Gromov-Hausdorff topology in the current  circumstances).  Therefore ii) implies iii).

 If we view iii) as the assumption, then Theorem \ref{thm1} and Proposition \ref{prop1} show that the origin $0\in \Delta$ is at finite  Weil-Petersson distance from $\Delta^{*}$. We obtain ii) by Theorem 1.3 in \cite{Ta}.
 \end{proof}

 Finally, we collect some earlier results mainly in \cite{Wang1,Wang2,To,Ta,RoZ1} for the  reader's convenience, and refer  readers to these papers for more detailed discussions.

  \begin{theorem}[\cite{Wang1,Wang2,Ta,To}]\label{thm3}   Let $(\pi: \mathcal{X}\rightarrow \Delta, \mathcal{L})$ be  a  degeneration  of polarized Calabi-Yau manifolds, and $\omega_{t}$ be  the unique Ricci-flat  K\"{a}hler-Einstein   metric
in  $
c_{1}(\mathcal{L})|_{X_{t}}\in H^{1,1}(X_{t}, \mathbb{R})$.
  Then the  following statements are equivalent.
 \begin{itemize}
  \item[i)]  The origin $0\in \Delta$ is at finite  Weil-Petersson distance from $\Delta^{*}$.
   \item[ii)]   The diameters $${\rm diam}_{\omega_t}(X_t)\leq D,$$ for a constant $D>0$ indpendent of $t$.
    \item[iii)] If we further assume that $\mathcal{X}\rightarrow \Delta$ comes from a quasi-projective family,  then by passing to a finite base change,  $\mathcal{X}\rightarrow \Delta$ is birational to a new family $\mathcal{X}'\rightarrow \Delta$ such that the new  central fiber $X_0'$ is a Calabi-Yau variety, and $\mathcal{X}\backslash X_0 \cong \mathcal{X}'\backslash X_0'$.   \end{itemize}
 \end{theorem}

  \begin{theorem} [\cite{Wang1,Wang2,To,Ta,RoZ1}]\label{thm4} Let $(\pi: \mathcal{X}\rightarrow \Delta, \mathcal{L})$ be a  degeneration of polarized Calabi-Yau $n$-manifolds such that $\mathcal{X}$ is normal, the relative canonical bundle $\mathcal{K}_{\mathcal{X}/\Delta}$ is trivial, i.e. $\mathcal{K}_{\mathcal{X}/\Delta}\cong \mathcal{O}_{\mathcal{X}}$, and $\mathcal{X}\rightarrow \Delta$ has at worst log-canonical singularities.
     Let  $\omega_{t}$ be  the unique Ricci-flat  K\"{a}hler-Einstein   metric
presenting  $
c_{1}(\mathcal{L})|_{X_{t}}\in H^{1,1}(X_{t}, \mathbb{R})$, $t\in\Delta^{*}$.
 Then the  following statements are equivalent.
 \begin{itemize}
  \item[i)]  The origin $0\in \Delta$ is at finite  Weil-Petersson distance from $\Delta^{*}$.
   \item[ii)] The central fiber  $X_0$ is a Calabi-Yau variety.
    \item[iii)]  When $t\rightarrow 0$,    $$(X_{t}, \omega_{t})
\stackrel{d_{GH}}\longrightarrow (Y, d_{Y}),  $$ in the Gromov-Hausdorff sense, where    $(Y, d_{Y}) $ is a compact metric space.  \end{itemize}
 \end{theorem}

\end{document}